\documentclass[12pt,a4paper]{article}
\usepackage{enumerate}
\usepackage{amsmath,amsthm}
\usepackage{amsfonts}
\usepackage{amssymb}
\usepackage{hyperref}
\usepackage{xcolor}
\newtheorem{theorem}{Theorem}[section]

\newtheorem{cor}[theorem]{Corollary}
\theoremstyle{definition}
\newtheorem{mydef}[theorem]{Definition}
\newtheorem{exa}[theorem]{Example}
\theoremstyle{remark}

\newcommand{\bigslant}[2]{{\raisebox{0em}{$#1$}\!\left/\raisebox{-.2em}{$#2$}\right.}}

\begin{document}

\title{  ON THE QUOTIENT SPACE OF A GENERALIZED ACTION OF A GENERALIZED GROUP\\  \vspace{1cm}{\small HASAN MALEKI, MOHAMMADREZA MOLAEI}\\ {\small Department of Mathematics, Shahid Bahonar University of Kerman, Kerman, Iran}\\ {\small e-mails: maleki\_hasan@yahoo.com; mrmolaei@uk.ac.ir }  }
\author{}
\date{}
\maketitle

\begin{abstract}
 In this paper by using of generalized groups and their generalized actions, we define and study the notion of $T$-spaces. Moreover, we study properties of the quotient space of a $T$-space and we present the conditions that imply to the Hausdorff property for it. Moreover, we study the maps between two $T$-spaces and we consider the notion of $T$-transitivity.
\end{abstract}

\section{Introduction}
 Groups appears in many branches of mathematics such as number theory, geometry and theory of lie groups and also they can find plenty applications even in physics and chemistry. Generalized groups are an interesting and fascinating extension of groups. They have been introduced in 1999 \cite{GG}. In a group, there is a unique identity element but in a  generalized group, each element has a spacial identity element. This kind of structure appears in genetic codes \cite{MS}. The role of identity as a mapping made many new important challenges. Generalized groups have been applied in genetic \cite{A, MS}, geometry \cite{TS} and dynamical systems \cite{CS}. The notion of generalized action \cite{m} is an extension of the notion of group actions \cite{Do, Gu}.  This notion has been studied first in 1999 \cite {m, TGP, MS}. The generalized action has been applied by the other researchers \cite{Neda eb, AGG}.\\
Let us recall the definition of a generalized group or completely simple semi group \cite{JAJK}.
A \textit{generalized group} is a semigroup  $ T $ with the following conditions:
\begin{enumerate}[(i)]
\item For each $t \in T$ there exists a unique $e(t) \in T$ such that $t\cdot e(t)=e(t)\cdot t=t$;
\item  For each $t \in T$ there exists $s \in T$ such that $s\cdot t=t\cdot s=e(t)$.
\end{enumerate}

One can easily prove that:  each $t$ in $T$, has a unique
inverse in $T$. The inverse of $t$ is denoted by $t^{-1}$. Moreover; for given $t \in T$, $e(t)=e(t^{-1})$ and $e(e(t))=e(t)$. Let $T$ and $S$ be tow generalized groups. A map $f : T \rightarrow S$ is called a \textit{homomorphism} if $f(st) = f(s)\cdot f(t)$ for every $s, t\in T$. We refer to \cite{MS} for more details. The geometrical viewpoint of completely simple semigroups are taken in \cite{OTS}. Here we study more properties of the generalized action of a generalized group on a topological spaces. In the next section we present the notion of $T$-spaces. We consider the quotient space of a $T$-space and we deduce the conditions that imply to the Hausdorff property for it. In the last section the maps between two $T$-spaces and $T$-transitivity are considered.

\section{$T$-Spaces}
Let us to begin this section by recalling the definition of topological generalized group.\\
\begin{mydef}
\label{tgg}
\cite{A, TGP} A \textit{topological generalized group} is a Hausdorff topological space $T$ which is endowed with a generalized group structure such that the generalized group operations $ m_{1}:T\rightarrow T $ defined by $t\mapsto t^{-1} $ and $ m_{2}:T \times T \rightarrow T $ by $ (s,t)\mapsto s\cdot t $ are continuous maps and  \begin{align}
 \label{st}
e(s\cdot t)=e(s)\cdot e(t).
\end{align}
 \end{mydef}

 \noindent As Shown in \cite{MS}, every topological generalized group $T$ is a disjoint union of topological groups. Moreover, the mapping $e:T\rightarrow T$ defined by $t\mapsto e(t)$, is a continuous map. Note that in a topological generalized group, the property~\eqref{st} implies that
\begin{align}
\label{sts}
e(s)\cdot e(t)\cdot e(s)=e(s), ~ \forall t,s \in T.
\end{align}
This property will be used frequently.

\begin{theorem}
If $T$ is a topological generalized group,then $$(s\cdot t)^{-1}=e(s)\cdot t^{-1}\cdot s^{-1}\cdot e(t)$$
\end{theorem}
\begin{proof}
 \begin{align*}
(s\cdot t)(e(s)\cdot t^{-1}\cdot s^{-1}\cdot e(t))&=s\cdot (t\cdot e(s)\cdot t^{-1}\cdot s^{-1}\cdot e(t))\\
&= s\cdot (t\cdot e(t)\cdot e(s)\cdot e(t)\cdot t^{-1}\cdot s^{-1}\cdot e(t))\\
&= s\cdot (t\cdot e(t)\cdot t^{-1}\cdot s^{-1}\cdot e(t))\\
&=s\cdot (e(t)\cdot s^{-1}\cdot e(t))\\
&=(s\cdot e(t)\cdot s^{-1})\cdot e(t)\\
&=(s\cdot e(s)\cdot e(t)\cdot e(s)\cdot s^{-1})\cdot e(t)\\
&=(s\cdot e(s)\cdot s^{-1})\cdot e(t)\\
&=e(s)\cdot e(t)=e(s\cdot t).
\end{align*}
We also have $(e(s)\cdot t^{-1}\cdot s^{-1}\cdot e(t))\cdot (s\cdot t)=e(s\cdot t)$. So $(s\cdot t)^{-1}=e(s)\cdot t^{-1}\cdot s^{-1}\cdot e(t)$.
\end{proof}

\begin{exa}
\label{e1}
If $T$ is the topological space $$  \mathbb{R}^{2}-\{(0,0)\}=\{re^{i\theta}|\quad r>0\quad \text{and}\quad 0\leqslant \theta< 2\pi\} $$ with the Euclidean metric, then $T$ with the multiplication
\begin{align}
(r_{1}e^{i\theta_{1}})\cdot(r_{2}e^{i\theta_{2}})=r_{1}r_{2}e^{i\theta_{2}}
\end{align}
is a topological generalized group. We have $e(re^{i\theta})=e^{i\theta}$ and $(re^{i\theta})^{-1}=\frac{1}{r}e^{i\theta}$. So we can see the identity set $e(T)$ is the unit circle $S^{1}$. However, $T$ is not a topological group.
\end{exa}

\begin{mydef}
\label{ga}
Suppose that $X$ is a topological space and $T$ is a topological generalized group. A \textit{generalized action} of $T$ on $X$ is a continuous map $\lambda:T\times X\longrightarrow X$ with the following properties:
\begin{enumerate}[(i)]
\item $\lambda(s,\lambda(t,x))=\lambda(s\cdot t,x)$, for $s,t \in T$ and $x\in X$;
\item If $x \in X$, then is $e(t) \in T$ such that $\lambda(e(t),x)=x$.
\end{enumerate}
\end{mydef}

Henceforth, $ \lambda(t,x) $ will be denoted by $tx$. For the more details on the generalized action, one can see \cite{AGG, m}. Now let us define the notion of $T$-spaces.
\begin{mydef}
A\textit{ $T$-space} is a triple $(X,T,\lambda)$ where $X$ is a Hausdorff topological space, $T$ is a topological generalized group and $\lambda:T\times X\longrightarrow X$  is   a   generalized action  of $T$ on $X$.
\end{mydef}

\begin{mydef}
If  $(X,T,\lambda)$ is a $T$-space and $x \in  X$, then:
\begin{enumerate}[(i)]
\item $ T_{x}=\{ t \in T| \quad tx=x \} $ is called the \textit{stabilizer} of $x$ in $T$;
\item $T(x)=\{tx | \quad t\in T\}$ is called \textit{$T$-orbit }of $x$ in $X$.
\end{enumerate}
\end{mydef}

For $x \in X$, $T_{x}$ is a generalized subgroup of $T$ \cite{AGG}. For $x \in X$, if the stabilizer $T_{x}$ is trivial (i.e. $T_{x}=\{e(t)\}$ for some $t \in T$), we say that $x$ is a \textit{regular} point, otherwise, it is called  a \textit{singular} point. The \textit{singular} set of $X$, denoted by $ \sum_{X} $, is the set of singular points of $X$. We define two maps $\theta_{t}:X\rightarrow X$ and $\rho_{x}:T\rightarrow X$, by $\theta_{t}(x)=tx$ and $ \rho_{x}(t)=tx $, respectively, where $t \in T$ and $x \in X$. $\theta_{t}$ and $\rho_{x}$ are continuous maps. Clearly, $T(x)=\rho_{x}(T)$ and  $T_{x} = (\rho_{x})^{-1}(x)$. So $T_{x} $ is a closed subset of $T$ and then we can say it is a closed generalized subgroup of $T$. Furthermore, $T_{x}$ is a topological generalized subgroup of $T$ \cite{AGG}.

By the action $\lambda$ of a $T$-space $(X,T,\lambda)$, we can define the following equivalence relation on $X$:
\begin{center}
$x\sim y$ if and only if there is $t \in T$ such that $tx=y$.
\end{center}
Now, we consider the quotient space $\bigslant{X}{\sim}$ and by the projection map $\pi:X\rightarrow \bigslant{X}{\sim}$, we define a topology on $\bigslant{X}{\sim}$ such that $\pi$ is a continuous map. In fact $U\subseteq \bigslant{X}{\sim}$ is open if $\pi^{-1}(U)$ is open in $X$.

Henceforth, we will use of the notation $\bigslant{X}{T}$ for the topological quotient space $\bigslant{X}{\sim}$.

\begin{mydef}
Suppose $(X,T,\lambda)$ is a $T$-space. Then  the generalized action $\lambda$ is called:
\begin{enumerate}[(i)]
\item \textit{effective }if for any two distinct $ s, t \in T$, there exists $x\in X$ such that $sx \neq tx$;
\item \textit{transitive} if for given $x,y\in X$, there exists  $t  \in T$ such that $tx=y$;
\item \textit{free} if for each $x\in X$, there exists precisely one $t \in T$ such that $tx=x$;
\item \textit{regular} if it is transitive and free.
\end{enumerate}
\end{mydef}

\noindent $\lambda$ is transitive if and only if $T(x)=X$, for each $x\in X$ and $\lambda$  is free if and only if $T_{x}=\{e(t)\}$, for each $x \in X$.

\begin{exa}
\label{e3}
Every topological generalized group  $T$ acts on itself by the multiplication of $T$ :  $st=s\cdot t$ for all $s,t \in T$. Note that this action need not be free, for instance, the action $xy=y$ as the multiplication of $T$ implies $T_{x}=T$, for $x \in X$.
\end{exa}

\begin{exa}
\label{e4}
Every topological generalized group  $T$ acts on itself by the multiplication $st=e(s)\cdot t.$
\end{exa}

\begin{exa}
Let $T$ be $\mathbb{R}-\{0\}$ with the Euclidean metric. $T$ with the multiplication $x\cdot y=x$ is a topological generalized group that if $x\in T$, then $e(x)=x^{-1}=x$. $T$ acts on itself with this multiplication. This action is regular.
\end{exa}

We recall that if $ T $ is a generalized group, $X$ is a set and $ S=\{\varphi^{t}~|\quad \varphi^{t}:X\rightarrow X \quad \text{is a mapping and } t \in T\} $, then the triple $(X,S,T)$ is called a \textit{ complete semidynamical system} if:
\begin{enumerate}[(i)]
\item $\varphi^{t_{1}}\circ \varphi^{t_{2}}=\varphi^{t_{1}\cdot t_{2}}$, for all $t_{1}, t_{2} \in T$;
\item For given $x \in X$, there is $\varphi^{t} \in D$ such that $x$ is a fixed point of $ \varphi^{t} $.
\end{enumerate}
We see that each $T$-space $(X,T,\lambda)$ generates a complete semidynamical system $(X,S,T)$ where
\begin{equation*}
S=\{ \theta_{t}:X\to X ~|~ \theta_{t}(x)=tx, \ \text{ for } x \in X  \text{ and }  t \in T  \}.
\end{equation*}

As shown in \cite{AGG}, if $e(T) \subseteq T_{x}$, for each $x\in X$, then $S$ with the multiplication $\theta_{s}\circ \theta_{t}=\theta_{st}$  is a topological generalized  group. In this case, for each $\theta_{t}\in S, e(\theta_{t})=\theta_{e(t)}$ and $(\theta_{t})^{-1}=\theta_{t^{-1}}$.

\begin{theorem}{\label{homeo}}
If $e(T) \subseteq T_{x}$,  for each $x\in X$, then  each $\theta_{t}$ is a homeomorphism.
\end{theorem}
\begin{proof}
 Every $\theta_{t}$ is a continuous map on $X$. Now we claim that each $\theta_{t}$ is one to one, onto and has the  inverse $\theta_{t^{-1}}$. If $\theta_{t}(x)=\theta_{t}(y)$, then $tx=ty$, and so $t^{-1}tx=e(t)x=x=y=e(t)y=t^{-1}ty$. Hence $\theta_{t}$ is one to one. On the other hand, for $x\in X$, there exists $t^{-1}x \in X$ such that $\theta_{t}(t^{-1}x)=tt^{-1}x=e(t)x=x$. Thus $\theta_{t}$ is also onto. If  $t\in T$ and $x\in X$, then
\begin{align*}
\theta_{t}\circ \theta_{t^{-1}}(x)= \theta_{tt^{-1}}(x)=\theta_{e(t)}(x)=e(t)x=x.
\end{align*}
The last equity follows from the fact $e(T) \subseteq T_{x}$. In the  same way
\begin{align*}
\theta_{t_{-1}}\circ \theta_{t}(x)= x.
\end{align*}
Therefore, every $\theta_{t}$ is a homeomorphism.
\end{proof}

\begin{theorem}\label{cm}
Let $(X,T,\lambda)$ be a $T$-space. If $T$ is compact and $e(T) \subseteq T_{x}$ for each $x\in X$, then $\lambda:T\times X\rightarrow X$ is a closed map.
\end{theorem}
\begin{proof}
Assume $C$ is a closed subset of $T\times X$ and $x \in X$ is a limit point of $\lambda(C)$. So  there exists a sequence $\{(t_{i},x_{i})\}$ in $C$ such that $\lambda(t_{i},x_{i})=t_{i}x_{i}$ converges to $x$. As $T$ is compact, then there is a subsequence of $\{t_{i}\} $ such that converges to a $t$ in $T$. We rename that subsequence be $\{t_{i}\} $. $T$ is a topological generalized group, so the map $ m_{1}:T\rightarrow T $,  defined by $t\mapsto t^{-1} $ is continuous. This implies that $\{t_{i}^{-1}\} $ converges to $t^{-1}$. $\lambda$ is also continuous, $\lambda(t_{i}^{-1},t_{i}x_{i})$ converges to $\lambda(t^{-1},x)$, so $\{e(t_{i})x_{i}\}$ converges to $t^{-1}x$. But for each $x\in X$, $e(T) \subseteq T_{x}$, thus $e(t_{i})x_{i}=x_{i}$ and consequently, $x_{i}$ converges to $t^{-1}x$. So the sequence $\{(t_{i},x_{i})\} $ in $C$ converges to $(t,t^{-1}x)$. Since $C$ is a closed subset of $T\times X$, then $(t,t^{-1}x) \in C$. Therefore, $\lambda(t,t^{-1}x)=e(t)x \in \lambda(C)$. According to the assumption, $e(t)x=x$. So $x\in \lambda(C)$ which means that $\lambda(C)$ is closed, that is, $\lambda$ is a closed map.
\end{proof}

Let $(X,T,\lambda)$ be a $T$-space. If $Y$ is a subset of $X$, then $$TY:=\{ty|\quad t\in T \quad \text{and} \quad y\in Y\}.$$ $Y$ is called \textit{invariant} under $T$ if $TY=Y$.

\begin{cor}\label{coro2}
Let $(X,T,\lambda)$ be a $T$-space. Moreover $T$ be compact, $Y\subseteq X$ and for each $x\in X$, $e(T) \subseteq T_{x}$, then
\begin{enumerate}[(i)]
\item  $TY$ is closed if $Y$ is closed;
\item  $TY$ is compact if $Y$ is compact.
\end{enumerate}
\end{cor}

\begin{theorem}
\label{om}
Let $(X,T,\lambda)$ be a $T$-space and for each $x\in X$, $e(T) \subseteq T_{x}$. Then the projection map $\pi:X\rightarrow \bigslant{X}{T}$ is an open map.
\end{theorem}
\begin{proof}
To prove that $\lambda $ is an open map, suppose $Y \subseteq X$ is an arbitrary open subset of $X$. We have \begin{equation}
\label{eq2}
 TY=\{ty~|~t\in T \text{ and }y\in Y\}= \pi^{-1}(\pi(Y))=  \bigcup_{t \in T} tY.
\end{equation}
Since  $\theta_{t}$ is a homeomorphism, then it is an open map. So for each $t\in T$, $\theta_{t}(Y)=tY$ is an open set in $X$. Hence $ \pi^{-1}(\pi(Y)) $  is an open set in $X$  (from \eqref{eq2}). So $\pi(Y)$ is open in $ \bigslant{X}{T}$, which means that $\pi$ is an open map.
\end{proof}

\begin{mydef}
Suppose  $(X,T,\lambda)$ is a $T$-space. A generalized action $\lambda$ is called \textit{perfect} if $e(T) \subseteq T_{x}$ for each $x\in X$. In this sense, the $T$-space is called \textit{perfect}.
\end{mydef}

Suppose  that $X$ and $Y$ are two topological spaces. A mapping $f:X\rightarrow Y $ is called proper if for each compact subset $A$ of  $Y$, $f^{-1}(A)$ is a compact subset of $X$ \cite{Topol}. We know that if the mapping $f:X\rightarrow Y $ is closed and  $f^{-1}(y)$ is compact for each $y\in Y$, then $f$ is a proper map \cite{Topol}.
\begin{theorem}
Let $(X,T,\lambda)$ be a $T$-space,  $T$ be compact  and  $ \lambda $ be perfect. Then
\begin{enumerate}[\bf (i)]
\item The projection map $\pi:X\rightarrow \bigslant{X}{T}$ is a closed map;
\item The quotient space $ \bigslant{X}{T}$ is Hausdorff;
\item The projection map $\pi:X\rightarrow \bigslant{X}{T}$ is a proper map;
\item $X$ is compact if and only if $\bigslant{X}{T}$ is compact;
\item  $X$ is  locally compact if and only if $\bigslant{X}{T}$ is  locally compact.
 \end{enumerate}
 \end{theorem}

 \begin{proof}
\begin{enumerate}[(i)]

\item Suppose that $Z \subseteq X$ is an arbitrary closed subset of $X$. As $ \lambda $ is perfect and $T$ is compact,  according to Theorem  \ref{coro2}, $ TZ= \pi^{-1}(\pi(Z)) $ is closed in $X$. Since $\pi$ is surjective map, we  can see that \begin{align}
\label{eq1}
\pi^{-1}(\bigslant{X}{T}-\pi(Z))= X- \pi^{-1}(\pi(Z)).
\end{align}
Since $ \pi^{-1}(\pi(Z)) $ is closed in $X$, then $X- \pi^{-1}(\pi(Z)) $ is open in $X$.  $ \bigslant{X}{T}-\pi(Z) $ is open in $\bigslant{X}{T}$. Thus $\pi(Z)$ is closed in $\bigslant{X}{T}$, hence $\pi$ is a closed map.
\item Suppose that $[x]$  and $ [y]$ are two distinct elements in $\bigslant{X}{T}$. Clearly, we can see that $T(x)\cap T(y)=\varnothing$. As mentioned, $T(x)=\rho_{x}(T)$, that is, the $T$-orbit $T(x)$ is the image of $T$. But we know that $\rho_{x}$ is continuous and $T$ is compact, which imply that  $T(x)=\pi^{-1}([x])$ is compact in $X$. In the same way, $T(y)$ is compact. Since $X$ is Hausdorff and $T(x)$ is compact in $X$ and $y\notin T(x)$, thus there exist disjoint open subsets $U$ and $V$ of $X$ containing $T(x)$  and $y$, respectively \cite{Topol}. Especially, it is  easy to see that $ T(x) \cap \overline{V} = \varnothing$. So $[x]=\pi(x)\notin \pi (\overline{V})$. On the other hand, $\pi (V)$ is open in $\bigslant{X}{T}$. $\pi (V)$ contains of $[y]$.   Moreover; since $\overline{V}$ is closed, $\pi (\overline{V})$ is closed in $\bigslant{X}{T}$. Thus $\bigslant{X}{T} - \pi (\overline{V}) $ is an open subset that contains $[x]$. Therefore $\bigslant{X}{T} - \pi (\overline{V}) $ and $\pi (V)$  are disjoint open sets of $\bigslant{X}{T}$ containing $[x]$  and $[y]$, respectively.
\item In (i), we saw that $\pi$ is closed map and also in (ii), we shown that $T(x)=\pi^{-1}([x])$ is compact for each $x \in x$. So $\pi$ is a proper map.
\item This follows easily from  continuity of $\pi$ and (iii).
\item This follows from Theorem \ref{om} and (iii).
\end{enumerate}\end{proof}

\begin{mydef}
Let $(X,T,\lambda)$ be a $T$-space. The generalized action $\lambda$ is called \textit{proper} if the map $\hat{\lambda}: T\times X\to X\times X$ defined by $\hat{\lambda}(t,x)=(tx,x)$ is a proper map.
\end{mydef}

\begin{theorem}
Let $(X,T,\lambda)$ is a $T$-space. The generalized action $\lambda$ is proper if and only if for every compact subset $Y$ of $X$, the set $Y_{T}=\{t\in T~|~ tY\cap Y \neq \varnothing\}$ is compact.
\end{theorem}

\begin{proof}
Suppose that the mapping $\hat{\lambda}$ is proper and $Y\subset X$ is compact. We see that $Y_{T}=P(\hat{\lambda}^{-1}(Y\times Y))$ where $P:T\times X\to T$ is the projection map. So $Y_{T}$ is compact. Conversely, Suppose that $Y_{T}$ is compact, where $Y$ is a compact subset of $X$. If $Z\subset X\times X$ is compact, and $Y:=\pi_{1}(Z)\cup\pi_{2}(Z)\subset X$, where $\pi_{1},\pi_{2}:X\times X\to X$ are the projections on the first and second components, respectively, then $Y$ is compact. Moreover $\hat{\lambda}^{-1}(Z)\subset \hat{\lambda}^{-1}(Y\times Y)\subset Y_{T}\times Y$. Since $ X\times X$ is Hausdorff and $Z$ is compact, then $Z$ is  a closed subset of $ X\times X$. So $\hat{\lambda}^{-1}(Z)$ is closed by continuity and it is also a closed subset of the compact set $Y_{T}\times Y$. Thus $Z$ is compact and $\lambda$ is proper.
\end{proof}
\begin{cor}
Let $(X,T,\lambda)$ be a $T$-space. If $T$ is compact, then the generalized action $\lambda$ is proper.
\end{cor}

\begin{cor}
Let $(X,T,\lambda)$ be a $T$-space. If $\lambda$ is proper, then
\begin{enumerate}[(i).]
\item The stabilizer $T_{x}$ is compact, where $x\in X$;
\item The orbit map $\rho_{x}$ is proper, where $x\in X$.
\end{enumerate}
\end{cor}

\begin{cor}
Let $(X,T,\lambda)$ be a $T$-space. If $\lambda$ is proper, then $T(x)$ is closed subset of $X$ and the quotient space $\bigslant{X}{T}$ is Hausdorff.
\end{cor}

\begin{exa}
Let $X= \mathbb{R}^2$ and let $T$ be the generalized group of Example \ref{e1}  which acts on $ X$ by
$$ (r_{1}e^{i\theta_{1}}).(r_{2}e^{i\theta_{2}})=r_{1}r_{2}e^{i\theta_{2}}. $$
The equivalence class $[x]$ is the set all $re^{i\theta_{0}}$ such that $r>0$, where $x=r_{0}e^{i\theta_{0}}\in \mathbb{R}^2-\{(0,0)\}$  and $[(0,0)]=\{(0,0)\}$. Moreover, for $x=r_{1}e^{i\theta_{1}}$ and $y=r_{2}e^{i\theta_{2}}$, $[x]=[y]$ if and only if $\theta_{1}=\theta_{2}$. So $Y=\bigslant{X}{T}\simeq S^{1}\bigcup \{(0,0)\}$ which is not a Hausdorff space with the quotient topology.

\end{exa}

\section{Maps On $T$-spaces}

Now, let us consider the maps of $T$-spaces. First, we recall the notion of transitivity of maps on a topological space. Let $X$ be a topological space and let $f:X\rightarrow X$  be a continuous mapping. Then the mapping $f$ is called \textit{ (topologically) transitive} [11] if for every pair of non-empty open subsets $U$ and $V$ of X, there is some $n \in  \mathbb{N}$ such that $f^{n}(U)\cap V \neq \varnothing$. Moreover, if $X$ is compact, then $f$ is transitive if and only if $f$ is onto and there is a point in $X$ with dense orbit \cite{transitivity}.
\begin{mydef}
Let $(X,T,\lambda)$ be a $T$-space. Suppose $f:X\rightarrow X$ is a continuous map. The map $f$ is called \textit{$T$-transitive} if for every pair of non-empty open subsets $U$ and $V$ in $X$, there is some positive integer $n$ and $t\in T$ such that $tf^{n}(U)\cap V \neq \varnothing$.
\end{mydef}

We can see that if $\lambda$ is not trivial and $f$ is a transitive map, then $f$ is $T$-transitive. But the fallowing example shows that each $T$-transitive map need not be transitive.
\begin{exa}
Let $X$ be  the topological space $[-1,1]$ with the topology generated by  Euclidean Metric. The topological generalized group $T=\{\pm 1 \}$ with the multiplication $s\cdot t=s|t|$ acts on $X$ by the generalized action $tx=t|x|$, where $t\in T$ and $x\in X$.  Then $f:X\rightarrow X $ is defined by
$$f(x)= \left \{ \begin{array}{lll}
  \sqrt{x} & \quad \text{if}& 0 \leq x \leq 1  \\
 -\sqrt{|x|} & \quad \text{if}& -1 \leq x \leq 0
  \end{array} \right.$$
We can see that $f$ is $T$-transitive but  it is not transitive.
\end{exa}

\begin{mydef}
Let $(X,T,\lambda)$ and $(Y,T,\mu)$ be two $T$-spaces. A continuous map $f:X\rightarrow Y$ is called
\begin{enumerate}[(i)]
\item \textit{$T$-equivariant} if $f(\lambda(t,x))=\mu(t,f(x))$ for $t\in T$ and $x\in X$, briefly, $f(tx)=tf(x)$;
\item \textit{$T$-pseudoequivariant} if $f(T(x))=T(f(x) )$ for $x\in X$.
\end{enumerate}
\end{mydef}
If $f:X\rightarrow Y$ is a $T$-equivariant between two $T$-spaces $X$ and $Y$, then $T_{x}\subseteq T_{f(x)}$  for each $x\in X$. Moreover, an $T$-equivariant map is clearly $T$-pseudoequivariant but the converse is not true. The following example shows this.
\begin{exa}
Let $X$ and $Y$ be two Euclidean space $\mathbb{R}$ and  $T=\{\pm 1\} $. Then $T$ with the multiplication $s\cdot t=s|t|$ is a topological generalized group. We define actions $\lambda,\theta:T\times \mathbb{R}\rightarrow \mathbb{R}$ by $\lambda(t,x)=t|x|$ and $\theta(t,x)=tx$, respectively. Now, we see that the identity map on $\mathbb{R}$ is a $T$-pseudoequivariant map which is not $T$-equivariant.
\end{exa}
\begin{theorem}
Let $(X,d)$ be a compact metric space with no isolated point and let $(X,T,\lambda)$ be a  $T$-space which $\lambda$ is perfect. Suppose $f:X\rightarrow X$ be a $T$-pseudoequivariant onto map. Then $f$ is $T$-transitive if and only if there exists $x \in X$ such that $\{ t\cdot f^n(x)~|~t\in T, n>o\}$ is dense in $X$.
\end{theorem}

\begin{proof}

Suppose $f$ is $T$-transitive. By using of Baire category theorem and the fact that $\lambda$ is perfect, we can see that there is some $x$ such that $\{ t\cdot f^n(x)~|~t\in T, n>o\}$ is dense in $X$. Conversely, suppose that $\{ tf^n(x)~|~t\in T, n>o\}$ is dense in $X$ for some $x\in X$. Let $U$ and $V$ be two non-empty open subsets of $X$. Then there are $s, t \in T$ and $m, n \in \mathbb{N}$ such that $sf^{m}(x)\in U$ and  $tf^{n}(x)\in V$. We assume that $m < n$. Since $\lambda$ is perfect then  $x\in f^{-m}(s^{-1}U)$. Thus $f^{n}(x)\in f^{n-m}(s^{-1}U)$. Moreover, since $f$ is $T$-pseudoequivariant, then there is $r\in T$ such that $ f^{n-m}(s^{-1}U)= rf^{n-m}(U)$. So $f^{n}(x) \in rf^{n-m}(U)$ and  $tf^{n}(x) \in trf^{n-m}(U)$. Therefore, $tf^{n}(x) \in t^{'}f^{p}(U) \cap V  $ where $t^{'}=tr\in T$ and $p=n-m \in \mathbb{N}$. So $f$ is $T$-transitive.
\end{proof}

 It may be that $T$-transitivity is not preserved by topological conjugacy. Now we introduce topological $T$-conjugacy that can preserve $T$-transitivity.

 \begin{mydef}
 Let $(X,T,\lambda)$ and $(Y,T,\theta)$ be two $T$-spaces. Moreover let $f:X\to X$ and $g:Y\to Y$ be two continuous maps. We say $f$ is \textit{topologically $T$-conjugate} to $g$ if there is a $T$-pseudoequivariant homeomorphism $h:X\to Y$ such that $h\circ f=g\circ h $.
 \end{mydef}

 \begin{theorem}
 Let $(X,T,\lambda)$ and $(Y,T,\theta)$ be two perfect $T$-spaces. If the continuous mapping $f:X\to X$ is topologically $T$-conjugate to  the continuous mapping $g:Y\to Y$ then $f$ is $T$-transitive if and only if $g$ is $T$-transitive.

 \end{theorem}

Now, we recall that a \textit{top sapce} $T$ is a  Hausdorff d-dimensional differentiable manifold which is also a topological generalized group such that mappings $m_{1}$ and $m_{2}$ are smooth \cite{TS}.
\begin{theorem}
Let $X$ and $Y$ be smooth manifolds and let $T$ be a top space. Suppose $F:X \rightarrow Y$ is a smooth map that is equivariant respect to a transitive perfect smooth generalized action $\lambda$ of $T$ on $X$ and a perfect smooth generalized action $\theta$ of $T$ on $Y$. Then $F$ has constant rank.
\end{theorem}

\begin{proof}
Let $x \in X$. Since $ \lambda$ is transitive then for any $y \in X$, there is some $t\in T$ such that $tx=y$. Since $\theta_{t}\circ F= F\circ \lambda_{t}$ for each $t \in T$, then we have  ${\theta_{t}}_{\ast}\circ F_{\ast}= F_{\ast}\circ {\lambda_{t}}_{\ast}$.
The generalized  actions $\lambda$ and $\theta$ are perfect, then ${\lambda_{t}}_{\ast}$  and ${\theta_{t}}_{\ast}$ are isomorphism. So the rank  $F$ at $y$ is the same as its rank at $x$. Thus $F$  has constant rank.
\end{proof}

\end{document}